\documentclass[12pt,reqno]{amsart}
\usepackage{amsmath,amssymb,amsthm,amscd,wasysym,mathrsfs,amsxtra,mathtools}
\input xy
\xyoption{all}

\usepackage[margin=2.3cm]{geometry}
\usepackage{enumitem}
\setlist[enumerate,1]{font=\normalfont, label=(\roman*)}

\usepackage{hyperref}

\usepackage{tikz}
\usetikzlibrary{backgrounds}

\newtheorem{theorem}{Theorem}[section]

\newtheorem{lemma}[theorem]{Lemma}

\newtheorem{proposition}[theorem]{Proposition}

\theoremstyle{definition}
\newtheorem{definition}[theorem]{Definition}
\newtheorem{problem}[theorem]{Problem}

\newtheorem{remark}[theorem]{Remark}

\newcommand{\IC}{\mathbb{C}}
\newcommand{\IR}{\mathbb{R}}
\newcommand{\IZ}{\mathbb{Z}}
\newcommand{\IL}{\mathbb{L}}
\newcommand{\IQ}{\mathbb{Q}}

\DeclareMathOperator{\Hom}{Hom}

\title[Quantum Riemann-Hilbert problems for the resolved conifold]{Quantum Riemann-Hilbert problems for the resolved conifold}  
\author[Wu-yen Chuang]{Wu-yen Chuang}
\address{Department of Mathematics and TIMS, National Taiwan University, Taipei, Taiwan}
\email{wychuang@gmail.com}


\begin{document}

\begin{abstract} 
We study the quantum Riemann-Hilbert problems determined by the refined Donaldson-Thomas theory on the resolved conifold. 
Using the solutions to classical Riemann-Hilbert problems in \cite{Bri2} we give explicit solutions in terms of multiple sine functions with unequal parameters. The new feature of the solutions is that the valid region of the quantum parameter $q^{\frac{1}{2}}=\exp(\pi i \tau)$ varies on the space of stability conditions and BPS $t$-plane. 
Comparing the solutions with the partition function of refined Chern-Simons theory and invoking large $N$ string duality, we find that the solution contains the non-perturbative completion of the refined topological string on the resolved conifold. Therefore solving the quantum Riemann-Hilbert problems provides a possible non-perturbative definition for the Donaldson-Thomas theory.  
\end{abstract}

\maketitle

\section{Introduction}
A BPS structure was first introduced in \cite{Bri1} to describe the Donaldson-Thomas theory on a three-dimensional Calabi-Yau category with stability conditions. For a variation of BPS structures with a natural growth condition a class of Riemann-Hilbert problems was proposed. They involve finding piecewise holomorphic maps from the complex plane into an algebraic torus with prescribed discontinuities along given BPS rays. Later Bridgeland gave a detailed solution to the Riemann-Hilbert problems for the resolved conifold, using a class of special functions related to Barnes' multiple Gamma and sine functions \cite{Bri2}. 

In \cite{BBS} quantum Riemann-Hilbert problems were formulated in terms of refined Donaldson-Thomas invariants. These involve piecewise holomorphic maps from the complex plane to the group of automorphisms of a quantum torus algebra. 

On the other hand, there have appeared large amounts of literature studying the Donaldson-Thomas theory on the resolved conifold. Not intending to give a complete list, here we simply name a few of them: \cite{ChuJaf}\cite{NN}\cite{Sze} for the unrefined theory, and \cite{DimGuk}\cite{MMNS} for the refined theory.

Motivated by the recent progress in the (quantum) Riemann-Hilbert problems and the results about the Donaldson-Thomas theory on the resolved conifold, in this paper we study its quantum Riemann-Hilbert problems and give a solution under certain conditions on the quantum parameter.

The outline is as follows. In the section \ref{BPS}, we review the refined BPS structures arising from the refined Donaldson-Thomas theory, and discuss the case of the resolved conifold. 
In section \ref{qtorus} we introduce the quantum torus algebra and the quantum dilogarithm functions, needed for the BPS automorphism. 
In section \ref{qRHp} we formulate the quantum Riemann-Hilbert problem. In section \ref{sol} we first introduce the multiple sine functions, study their asymptotic expansions near $0$ and $\infty$, and then give a solution to the quantum Riemann-Hilbert problem for the resolved conifold. 

The new feature of the solutions is that the valid region of the quantum parameter $q^{\frac{1}{2}}=\exp(\pi i \tau)$ varies on the space of stability conditions and BPS $t$-plane. 
It is not clear to us whether the restriction of $\tau$ is due to the limitation of our approach or it has any physical implication. 

Finally we compare the solutions with the partition function of refined Chern-Simons theory, and find that the solutions contain the non-perturbative completion of the refined topological string/Gromov-Witten theory on the resolved conifold after invoking the large $N$ duality in string theory. Therefore solving the quantum Riemann-Hilbert problems provides a possible non-perturbative definition for the Donaldson-Thomas theory.

\medskip

\noindent{\bf Acknowledgement.} We would like to thank the referee whose comments helped improve greatly the presentation of the paper. WYC was partially supported by Taiwan MOST grant 109-2115-M-002-007-MY2 and NTU Core Consortiums grant 110L892103 and 111L891503 (TIMS).

\section{Refined BPS structures}\label{BPS}

\subsection{Definition}
In \cite{Bri1} the numerical Donaldson-Thomas theory was used to give the definition of a BPS structure, which is a special case of Kontsevich and Soibelman's notion of a stability structure. The following definition is the natural analogue for the refined Donaldson-Thomas theory, which first appeared in \cite{BBS}.

\begin{definition}
A refined BPS structure $(\Gamma, Z, \Omega)$ consists of the data
\begin{enumerate}
	\item A finite rank free abelian group $\Gamma \simeq \mathbb{Z}^{\oplus n}$, equipped with a skew-symmetric form $$\langle -,- \rangle : \Gamma \times \Gamma \to \IZ;$$
	\item A homomorphism of abelian group, called central charge,
	$$Z: \Gamma \to \IC;$$
	\item A map of sets 
	\begin{equation}
		\Omega: \Gamma \to \IQ[\IL^{\pm \frac{1}{2}}], \ \ \ \Omega(\gamma)=\sum_{n \in \IZ} \Omega_n(\gamma) \IL^{\frac{n}{2}},
	\end{equation} where $\IL^{\frac{1}{2}}$ is a formal symbol, satisfying the following two conditions:\\
	(a) Symmetry: $\Omega(-\gamma)=\Omega(\gamma)$ for all $\gamma \in \Gamma$, and $\Omega(0)=0;$\\
	(b) Support property: fixing a norm $|| \cdot ||$ on the finite dimensional vector space $\Gamma \otimes_{\IZ} \IR$, there is a constant $C>0$ such that
	$|Z(\gamma)|> C||\gamma||$ for $\gamma\in\Gamma$ with $\Omega(\gamma) \neq 0$.
\end{enumerate}
\end{definition}

The unrefined limit is recovered by taking $\IL^{\frac{1}{2}}=-1$ and $\Omega: \Gamma \to \IQ$. The support property is listed only for completeness and is satisfied in the resolved conifold case. 

\begin{definition}
The active BPS rays $\ell \subset \IC^{\ast}$ of the refined BPS structures are defined to be the rays $\ell=\IR_{>0} \cdot Z(\gamma)$ for $\gamma \in \Gamma$ with $\Omega(\gamma) \neq0.$	
\end{definition}

\subsection{Refined BPS structure for the resolved conifold}
Let $X$ be the resolved conifold, which is the total space of $\mathcal{O}_{\mathbb{P}^1}(-1)^{\oplus 2}$ over $\mathbb{P}^1$ and also the resolution of the ordinary double point singularity 
\begin{equation} 
	(x_1 x_2 -x_3 x_4 = 0) \subset \IC^4 .
\end{equation}

Let $\text{D}^b \text{Coh}(X)$ be the bounded derived category of coherent sheaves on $X$ and $\mathcal{D} \subset \text{D}^b \text{Coh}(X)$ be the full triangulated subcategory, consisting of complexes whose cohomology sheaves have support dimension $\leq 1$. The quantum Riemann-Hilbert problems we consider in this paper arise from the refined Donaldson-Thomas theory of the category $\mathcal{D}$. 
By \cite[Theorem A.2]{Bri2} the space of stability conditions on the derived category $\mathcal{D}$, quotiented by the subgroup of autoequivalences generated by spherical twists, is given by 
\begin{equation} 
	M_{stab} = \{ (v,w) \in \IC^2 \ |\ w \neq 0, v+nw \neq 0\ \text{for\ all}\ n \in \IZ \} \subset \IC^2 \ .
\end{equation}

We decompose $$M_{stab}=M_+\sqcup M_0\sqcup M_-$$ by the sign of $\text{Im}(v/w)$. In this paper we only study the non-degenerate region $M_+$, since the study of the regions $M_0$ and $M_-$ is completely analogous. 

Consider the Chern characters of compactly-supported sheaves $\mathcal{O}_C(n)$, $\mathcal{O}_x$ on $X$
and express them as
\begin{equation} 
	\text{ch}(\mathcal{O}_C(n))=\beta - n\delta, \ \text{ch}(\mathcal{O}_x)= -\delta\ .
\end{equation}

Given a point $(v,w) \in M_{stab}$ we have the following refined BPS structure.
\begin{enumerate}
	\item The charge lattice $\Gamma_{\leq 1} = \IZ \beta \oplus \IZ \delta$, with the zero skew-symmetric intersection form $\langle -,- \rangle =0$.
	\item The central charge is given by $Z_{\leq 1}(a\beta+b\delta)=2 \pi i (av+bw)\ .$
	\item The nonzero refined BPS invariants are 
\begin{equation}
		\Omega(\gamma)=\begin{cases} 1 &\text{if }\gamma=\pm \beta+n\delta \text{ for some } n\in \IZ,\\
			\IL^{\frac{1}{2}}+\IL^{-\frac{1}{2}} &\text{if }\gamma=k\delta\text{ for some }k\in \IZ\setminus\{0\}.\end{cases}
\end{equation}
\end{enumerate}	

The refined BPS invariants here can be read off from the Donaldson-Thomas partition function in \cite{DimGuk}\cite{MMNS}.
 
%The lattice $\Gamma_{\leq 1}$ is generated by the Chern characters of compactly-supported sheaves on $X$. We choose the convention such that
%\begin{equation} 
%	\text{ch}(\mathcal{O}_C(n))=\beta - n\delta, \ \text{ch}(\mathcal{O}_x)= -\delta\ .
%\end{equation}

In order to have a nontrivial quantum Riemann-Hilbert problem one needs to consider the doubled lattice
$$ \Gamma = \Gamma_{\leq 1} \oplus \Gamma_{\geq 2}, \ 
\Gamma_{\geq 2} := \Gamma_{\leq 1}^\vee=\text{Hom}_{\IZ}(\Gamma_{\leq 1}, \IZ). $$ 
In terms of dual basis $\beta^{\vee}$ and $\delta^{\vee}$, the lattice $\Gamma_{\geq 2}$ is given by,
\begin{equation}
\Gamma_{\geq 2} = \IZ \beta^\vee \oplus \IZ \delta^\vee \ .	
\end{equation} 
The skew-symmetric form $\langle -,- \rangle $ on $\Gamma$ is then defined by $\langle -,- \rangle \vert_{\Gamma_{\leq 1}} = \langle -,- \rangle \vert_{\Gamma_{\geq 2}}=0$
and $\langle \beta^{\vee}, \beta \rangle =\langle \delta^{\vee}, \delta \rangle =1$. 
Later we also write $\Gamma_e=\Gamma_{\leq 1}$ and $\Gamma_m=\Gamma_{\geq 2}$.
We extend both the maps $Z$ and $\Omega$ by zero on $\Gamma_{\geq 2}$.

\section{Quantum torus algebra and quantum dilogarithm}\label{qtorus}

\subsection{Definition of quantum torus algebra}
We define the quantum torus algebra to be 
\begin{equation}
\IC_q[\mathbb{T}] = \oplus_{\gamma\in\Gamma}\, \IC[\IL^{\pm\frac{1}{2}}] \ x_{\gamma}, \qquad x_{\gamma_1} \ast x_{\gamma_2} = \IL^{\frac{1}{2}\langle \gamma_1, \gamma_2 \rangle}\, x_{\gamma_1 + \gamma_2}.
\end{equation}

This is a quantization of the ring of functions on the algebraic torus $\mathbb{T}=\text{Hom}_{\IZ}(\Gamma, \IC^{\ast})$.

A quadratic refinement of the form $\langle-,-\rangle$ is an element of the finite set 
\begin{equation}
	\{ \sigma : \Gamma \to \{ \pm 1\} \ | \ \sigma(\gamma_1+\gamma_2) = 
	(-1)^{\langle\gamma_1,\gamma_2\rangle} \ \sigma(\gamma_1) \sigma(\gamma_2) \} \ .
\end{equation}

Using a quadratic refinement $\sigma: \Gamma \to \{\pm 1\}$ one can alternatively use $y_{\gamma} = \sigma(\gamma) x_{\gamma}$ as generators for $\IC_q[\mathbb{T}]$. Let $q^{\frac{1}{2}}= -\IL^{\frac{1}{2}}$. Now we have 
\begin{equation}
\IC_q[\mathbb{T}] = \oplus_{\gamma\in\Gamma}\, \IC[q^{\pm\frac{1}{2}}] \ y_{\gamma},\qquad
 y_{\gamma_1} \ast y_{\gamma_2} = q^{\frac{1}{2}\langle \gamma_1, \gamma_2 \rangle}\, y_{\gamma_1 + \gamma_2}.
\end{equation}

\subsection{Doubled and uncoupled case}
\begin{definition} 
A refined BPS structure $\big(\Gamma=\Gamma_e\oplus\Gamma_m, Z, \Omega\big)$ satisfying $\Gamma_m=\Gamma_e^\vee$, $\langle -,- \rangle|_{\Gamma_e} = \langle -,- \rangle|_{\Gamma_m} =0$ and $\Omega(\gamma_m)=0$ for all $\gamma_m \in \Gamma_m$ is said to be {\it doubled and uncoupled.}
\end{definition}

\begin{remark}
The refined BPS structure on the resolved conifold is doubled and uncoupled.
\end{remark}

Now we introduce the extended quantum torus algebra to deal with the doubled case. We define the extended quantum torus algebra to be the noncommutative algebra
\begin{equation}
	\widehat{\IC_q[\mathbb{T}]} = \oplus_{\gamma_m \in \Gamma_m}\, \mathcal{M}(\mathcal{H} \times V_e)\cdot y_{\gamma_m},   
\end{equation}
where $\mathcal{H}$ is the upper half-plane in the complex plane $\IC$, $V_e = \text{Hom}_{\IZ}(\Gamma_e, \IC)$, and $\mathcal{M}(\mathcal{H} \times V_e)$ is the field of meromorphic functions on the product space. 

The product $\widehat{\ast}$ is defined by 
\begin{equation}
	\big( f_1(\tau,\theta)\cdot y_{\gamma_{m1}} \big) \widehat{\ast} \big( f_2(\tau,\theta)\cdot y_{\gamma_{m2}} \big) = f_1(\tau,\theta-\langle \gamma_{m2}, - \rangle \tau/2)f_2(\tau,\theta + \langle \gamma_{m1}, - \rangle \tau/2) \cdot y_{\gamma_{m1}+\gamma_{m2}},
\end{equation} where $\tau \in \mathcal{H}$ and $\theta$ is an element in $V_e$. Notice that the definition of $\widehat{*}$ here is slightly different from the one in \cite{BBS}. The reason for this choice of $\widehat{*}$ is purely computational, since it gives manageable automorphisms associated with the active BPS rays. 

There is a commutative subalgebra
\begin{equation}
\widehat{\IC_q[\mathbb{T}]}_0 = \mathcal{M}(\mathcal{H} \times V_e) \cdot 1 \subset \widehat{\IC_q[\mathbb{T}]} \ .
\end{equation}

%\begin{remark}
%	Another slight difference is that here we consider $\mathcal{M}(\IC \times V_e)$ instead of $\mathcal{M}(\mathcal{H} \times V_e)$, where
%	$\mathcal{H}$ is the upper half plane. 
%	Later in our solution to the quantum Riemann-Hilbert problem, we indeed can restrict $\tau$ to a half-plane. However the orientation of the half-plane is not fixed. Therefore we use $\IC$ instead.  
%\end{remark}

\begin{lemma}\label{injI}
	If the refined BPS structure is doubled and uncoupled, the map $I:\IC_q[\mathbb{T}] \to \widehat{\IC_q[\mathbb{T}]}$ defined by 
\begin{equation}
I(q^{\frac{k}{2}} \cdot y_{\gamma_e+\gamma_m}) = {\rm exp} (\pi i \tau k + 2\pi i \theta(\gamma_e)) \cdot y_{\gamma_m}
\end{equation} is an injective ring homomorphism. 
\end{lemma}
\begin{proof}
It is easily verified that $I$ is a ring homomorphism. The injectivity of $I$ follows from the fact that the set $\big\{\text{exp}(2\pi i \theta(\gamma_e))\big\}_{\gamma_e\in\Gamma_e}$ is linearly independent over $\IC$.
\end{proof}

\subsection{Quantum dilogarithm}
The quantum dilogarithm function is given by
\begin{equation}\label{qdilog}
	\mathbb{E}_q(x) = \prod_{k \geq 0}(1-x q^k),
\end{equation} which converges absolutely for $|q| < 1$ and defines a nowhere-vanishing analytic function for $x \in \IC$. From the definition of the quantum dilogarithm function we know

\begin{equation}
	\mathbb{E}_q(x)\ \mathbb{E}_q(qx)^{-1}=1-x.
\end{equation} 

To each active BPS ray $\ell \subset \IC^*$ we attach a product 
\begin{equation}\label{BPSauto1}
	\text{DT}_q(\ell) = \prod_{Z(\gamma)\in \ell } \prod_{n\in \IZ}
	\mathbb{E}_q\big((-q^{\frac{1}{2}} )^{n+1}y_{\gamma}\big)^{-(-1)^n\Omega_n(\gamma)} ,
\end{equation}which could be an infinite product of quantum dilogarithm functions in general and is not an element in 
$\IC_q[\mathbb{T}]$ (see e.g. \cite{FGFS} for related discussions). 
Further assuming that the refined BPS structure is 
ray-finite, i.e. for any active ray $\ell$ there are only finitely many $\gamma \in \Gamma$ for which $Z(\gamma) \in \ell$ and $\Omega(\gamma) \neq 0$,
then the expression for $\text{DT}_q(\ell)$ (\ref{BPSauto1}) becomes a finite product of quantum dilogarithm functions.
In the region $\tau \in \mathcal{H}$, i.e. $|q| < 1$, each product factor in (\ref{BPSauto1}) is a nowhere-vanishing analytic function. So 
via the injective ring homomorphism $I$ in Lemma \ref{injI}, $\text{DT}_q(\ell)$ defines an element in  $\widehat{\IC_q[\mathbb{T}]}$,
\begin{equation}
I(\text{DT}_q(\ell)) \in \widehat{\IC_q[\mathbb{T}]}.
\end{equation}

Then the automorphism associated with the active BPS ray $\ell$ is defined to be 
\begin{equation}\label{BPSauto2}
\mathbb{S}_q(\ell)=\text{Ad}_{I(\text{DT}_q(\ell))} \in \text{Aut}\  \widehat{\IC_q[\mathbb{T}]} \ .
\end{equation} 

%The expressions for $\text{DT}_q(\ell)$ and 
%the automorphism $\mathbb{S}_q(\ell)$ in (\ref{BPSauto1})(\ref{BPSauto2}) apply to any refined BPS structures, not limited to the doubled refined BPS structures. When the refined BPS structure is doubled and uncoupled, 
%using Lemma \ref{injI} we can work with these automorphisms either in $\IC_q[\mathbb{T}]$ or $\widehat{\IC_q[\mathbb{T}]}$.

\begin{lemma}\label{autosql} Let $(\Gamma=\Gamma_e\oplus\Gamma_m, Z, \Omega)$ be a doubled and uncoupled refined BPS structure, which is also ray-finite. 
Let $\ell \subset \IC^*$ be an active BPS ray. Then the automorphism $\mathbb{S}_q(\ell)$ acts trivially on $I(y_{\gamma_e}) \in \widehat{\IC_q[\mathbb{T}]}$ for $\gamma_e \in \Gamma_e$ and the action of $\mathbb{S}_q(\ell)$ on $y_{\gamma_m}$ for $\gamma_m \in \Gamma_m$ is given by 
\begin{equation}\label{Sql}
\mathbb{S}_q(\ell)(y_{\gamma_m}) =\prod_{Z(\gamma)\in \ell}\  \prod_{n\in \IZ} \prod_{k=0}^{|\langle \gamma_m, \gamma \rangle | -1}
(1+(-q^{\frac{1}{2}})^n (q^{\frac{1}{2}})^{2k+1-|\langle \gamma_m, \gamma \rangle |} I(y_{\gamma}))^{(-1)^n\Omega_n(\gamma) \text{sgn}\langle \gamma,\gamma_m \rangle} \cdot y_{\gamma_m} \ ,
\end{equation} where $\text{sgn} \langle \gamma, \gamma_m \rangle$ is the sign of $\langle \gamma, \gamma_m \rangle$.
\end{lemma}
\begin{proof}
It is clear that $\mathbb{S}_q(\ell)$ acts trivially on $I(y_{\gamma_e}) \in \widehat{\IC_q[\mathbb{T}]}$ with $\gamma_e \in \Gamma_e$ due to the doubled and uncoupled structure.

For $y_{\gamma_m} \in \widehat{\IC_q[\mathbb{T}]}$ with $\gamma_m\in\Gamma_m$, we have
\begin{align}
&\mathbb{S}_q(\ell)(y_{\gamma_m})=\text{Ad}_{I(\text{DT}_q(\ell))}(y_{\gamma_m}) \nonumber \\
&=\prod_{Z(\gamma)\in \ell}\ \prod_{n\in \IZ} \mathbb{E}_q\big((-q^{\frac{1}{2}} )^{n+1}I(y_{\gamma})\big)^{-(-1)^n\Omega_n(\gamma)} 
\ \widehat{\ast}\  y_{\gamma_m}\ \widehat{\ast}\ \mathbb{E}_q\big((-q^{\frac{1}{2}} )^{n+1}I(y_{\gamma})\big)^{(-1)^n\Omega_n(\gamma)}\ . \nonumber \\
&= \prod_{Z(\gamma)\in \ell}\ \prod_{n\in \IZ} \mathbb{E}_{e^{2\pi i \tau}}\big((-e^{\pi i \tau} )^{n+1} e^{2\pi i \theta(\gamma)}\big)^{-(-1)^n\Omega_n(\gamma)} \ 
\widehat*\ y_{\gamma_m}\ \widehat*\ \mathbb{E}_q\big((-e^{\pi i \tau} )^{n+1}e^{2\pi i \theta(\gamma)}\big)^{(-1)^n\Omega_n(\gamma)} \nonumber \\
&= \prod_{Z(\gamma)\in \ell} \prod_{n\in \IZ} \mathbb{E}_{e^{2\pi i \tau}}\big((-e^{\pi i \tau})^{n+1}(e^{\pi i \tau})^{-\frac{\langle \gamma_m, \gamma \rangle}{2}} e^{2\pi i \theta(\gamma)}\big)^{-(-1)^n\Omega_n(\gamma)} \cdot \nonumber \\
&\qquad\qquad\qquad\qquad  \mathbb{E}_q\big((-e^{\pi i \tau} )^{n+1}(e^{\pi i \tau})^{\frac{\langle \gamma_m, \gamma \rangle}{2}}e^{2\pi i \theta(\gamma)}\big)^{(-1)^n\Omega_n(\gamma)} y_{\gamma_m}.
\end{align} 
Then using the definition of quantum dilogarithm (\ref{qdilog}) we obtain (\ref{Sql}).
\end{proof}

%\begin{remark}
%More precisely, we have 
%\begin{equation}
%	I\big(\mathbb{S}_q(\ell)(y_{\gamma_m})\big) =\prod_{Z(\gamma)\in \ell}\  \prod_{n\in \IZ} \prod_{k=0}^{|\langle \gamma_m, \gamma \rangle | -1}
%	\big(1+(-q^{\frac{1}{2}})^n (q^{\frac{1}{2}})^{2k+1-|\langle \gamma_m, \gamma \rangle |} \text{exp}(2 \pi i \theta(\gamma))\big)^{(-1)^n\Omega_n(\gamma) \text{sgn}\langle \gamma, \gamma_m \rangle} \cdot y_{\gamma_m} \ .
%\end{equation}
%\end{remark}

\begin{remark}
We may choose a quadratic refinement to absorb the signs in (\ref{Sql}). For the resolved conifold case, $\sigma$ is chosen such that $\sigma(\pm\beta)=-1, \sigma(\delta)=1$. This choice is consistent with \cite[Section 7.7]{GMN13}, in which they use $\sigma(\gamma_{hyper})=-1,\sigma(\gamma_{vector})=1$.

When $\gamma=\pm \beta + k \delta$, we only have $\Omega_0(\gamma) \neq 0$. When $\gamma=k \delta$, the nonvanishing invariants are
$\Omega_1(\gamma)$ and $\Omega_{-1}(\gamma)$.
In terms of the generators $x_{\gamma}$ with $\gamma\in\Gamma_e$ we have for the resolved conifold
\begin{equation}
\mathbb{S}_q(\ell)(x_{\gamma_m}) =\prod_{Z(\gamma)\in \ell}\  \prod_{n\in \IZ} \prod_{k=0}^{|\langle \gamma_m, \gamma \rangle | -1}
(1-(q^{\frac{1}{2}})^{n+2k+1-|\langle \gamma_m, \gamma \rangle |} I(x_{\gamma}))^{(-1)^n\Omega_n(\gamma) \text{sgn}\langle \gamma, \gamma_m \rangle} \cdot x_{\gamma_m} \ .    
\end{equation}
\end{remark}

\section{Quantum Riemann-Hilbert problems}\label{qRHp}

\subsection{Quantum Riemann-Hilbert problems}
Given a ray $\ell \in \IC^{\ast}$ the corresponding half-plane is given by
\begin{equation}
\mathcal{H}_{\ell} = \{ z\in \IC^{\ast} \ | \ z=uv \text{ with} \ u \in \ell \text{ and } \text{Re}(v) >0 \} \ .
\end{equation}

In the formulation of Riemann-Hilbert problems defined by the BPS structure of classical Donaldson-Thomas theory \cite[Problem 2.4]{Bri2}, we look for holomorphic function $\Phi_{\ell}:\mathcal{H}_{\ell} \to \mathbb{T}=\Hom_{\IZ}(\Gamma, \IC^*)$ and impose three conditions (RH1)(RH2)(RH3), such that all $\Phi_{\ell}$'s have 
jumping behaviors dictated by BPS automorphisms,
$0$ limit as $t \to 0$ and  polynomial growth at $t \to \infty$.

Motivated by these conditions, next we formulate the quantum Riemann-Hilbert problem for the doubled and uncoupled refined BPS structure.

\begin{problem}[Quantum Riemann-Hilbert problem]
	\label{qRHp1}
Let $(\Gamma=\Gamma_e\oplus\Gamma_m, Z, \Omega)$ be a doubled and uncoupled refined BPS structure, such that $\mathbb{S}_q(\ell)$ 
is well-defined for every active ray $\ell$.

For each non-active BPS ray $\ell \in \IC^*$ we look for maps 
\begin{equation}
\Phi_{\ell} : \mathcal{H}_{\ell} \to \text{Aut}\  \widehat{\IC_q[\mathbb{T}]} \ ,
\end{equation} satisfying the following three properties.

\begin{itemize} 
\item(qRH1) Suppose that two non-active rays $\ell_1, \ell_2 \in \IC^*$ in the clockwise order are the boundary rays of an acute sector $\Delta \subset \IC^*$, which contains finite active BPS rays $\tilde{\ell}_1, \cdots, \tilde{\ell}_n$. Then we have 
\begin{equation}
\Phi_{\ell_2}(t)= \Phi_{\ell_1}(t) \circ \mathbb{S}_q(\Delta) =  \Phi_{\ell_1}(t) \circ \mathbb{S}_q(\tilde{\ell}_1)\circ \cdots \circ \mathbb{S}_q(\tilde{\ell}_n),
\end{equation} 
for all $t \in \mathcal{H}_{\ell_1} \cap \mathcal{H}_{\ell_2}.$ \\

\item(qRH2) Let $\ell \in \IC^*$ be a non-active ray, whose closest active BPS rays in the anticlockwise and clockwise directions are $\ell_1$ and $\ell_2$ respectively. For each $\gamma = \gamma_e + \gamma_m \in \Gamma, \gamma_e \in \Gamma_e, \gamma_m \in \Gamma_m$ we have 
\begin{equation}\label{qRH2}
	\Phi_{\ell}(t) (I(x_{\gamma})) = \text{exp}\big(- \frac{Z(\gamma_e)}{t} + 2 \pi i \theta(\gamma_e)\big)\  R_{\ell,\gamma_m}(t,q^{\frac{1}{2}}) \cdot x_{\gamma_m} \in \widehat{\IC_q[\mathbb{T}]} ,
\end{equation}
where $q^{\frac{1}{2}}= \text{exp} (\pi i \tau)$ and $R_{\ell,\gamma_m}(t,q^{\frac{1}{2}})$ is a holomorphic function 
of $t \in \mathcal{H}_{\ell}$ and $\tau \in \mathcal{H}$, satisfying
$$ R_{\ell,\gamma_m}(t,q^{\frac{1}{2}}) \to 1,  $$ as $t\to 0$ in any closed subsector in $\mathcal{H}_{\ell_1} \cap \mathcal{H}_{\ell_2}$.

\item(qRH3) Let $\ell \in \IC^*$ be a non-active ray, whose closest active BPS rays in the anticlockwise and clockwise directions are $\ell_1$ and $\ell_2$ respectively. 
For each $\gamma_m \in \Gamma_m$ there exists $k>0$ such that 
\begin{equation}
|t|^{-k} < |R_{\ell,\gamma_m}(t,q^{\frac{1}{2}})| < |t|^k \ ,
\end{equation} for $t$ in any closed subsector of $\mathcal{H}_{\ell_1} \cap \mathcal{H}_{\ell_2}$ and $|t| \gg 0$.
\end{itemize}
\end{problem}

\begin{remark}
	In order to have a full-fledged quantum Riemann-Hilbert problem, in (qRH1) we
	also need to describe the BPS automorphism $\mathbb{S}_q(\Delta)$
	associated to the acute sector $\Delta$, containing infinitely many active 
	rays. The general discussion on the well-definedness of the classical BPS 
	automorphism $\mathbb{S}(\Delta)$ is provided in \cite[Appendix B]{Bri1}. 
	We leave the analogous analysis of the quantum counterpart 
	$\mathbb{S}_q(\Delta)$ for future study. 
	Although there exist such acute sectors $\Delta$ containing 
	infinitely many active rays on the resolved conifold, we will deal with 
	such cases by a more explicit approach later.  
	
	The $\text{exp}(-Z(\gamma_e)/t)$ part of the ansatz in (qRH2)  is motivated by \cite[Problem 2.4 (RH2)]{Bri2}.
	
	When solving $R_{\ell,\gamma_m}(t,q^{\frac{1}{2}})$ for all $\tau \in \mathcal{H}$ is not possible, we are allowed to relax the condition (qRH2) to restrict $\tau$ to an open set in $\mathcal{H}$.   
\end{remark}

%\begin{remark} In this version we have formulated in such a way that the asymptotic behaviors of the solutions to the quantum Riemann-Hilbert problems coincide with the classical ones. It is likely that the correct formulation of the asymptotics should involve the parameters $q^{\frac{1}{2}}$.
%\end{remark}

\subsection{Quantum Riemann-Hilbert problem for the resolved conifold}\label{qRHpconifold}
Now back to the resolved conifold case. Let $X$ be the resolved conifold and $Z(a\beta+b\delta)=2 \pi i (av+bw)$ be the central charge function associated with the point $(v,w)\in M_+ \subset M_{stab}$. The active rays consist of 
\begin{equation}
\pm \ell_{\infty} = \pm \IR_{>0} 2 \pi i w, \ \pm \ell_n = \pm \IR_{>0} 2 \pi i (v+nw) \in \IC^* \ .
\end{equation} 
Define $\Sigma(n)$ to be the convex open sector with boundary rays $\ell_{n-1}$ and $\ell_n$.

Since the rays $\pm \ell_{\infty}$ contain infinitely many active classes, the refined BPS structure of the resolved conifold 
is not ray-finite. Moreover an acute sector $\Delta$ containing $\ell_{\infty}$ also contains infinitely many active rays. 
However, the classical BPS structure of the resolved conifold satisfies a convergent condition, i.e. $\sum_{\gamma \in \Gamma} |\Omega(\gamma)|\exp(-R|Z(\gamma)|) < \infty$ for some $R>0$, defined in \cite[Definition 2.1]{Bri2}. Then by the analysis in \cite[Appendix B]{Bri1} and \cite[Proposition 2.2]{Bri2}, it is
proved in \cite[Sec. 3.2]{Bri2} that the classical BPS automorphism $\mathbb{S}(\ell_{\infty})$ exists on the analytic open subset $|\mathtt{x}_{\delta}|<1$ 
of the twisted torus.
And $\mathbb{S}(\Delta)$ is also well-defined on the analytic open subset 
$|\mathtt{x}_{\delta}|<1$. Here $\mathtt{x}_{\delta}$ is the twisted character in the classical 
Riemann-Hilbert problem. 
 
At the quantum level, in terms of the quantum twisted characters $x_{\gamma}$,
the BPS automorphism of $\widehat{\IC_q[\mathbb{T}]}$ in Lemma \ref{autosql} gives
\begin{align}\label{sql123}
 &\mathbb{S}_q(\ell_n) (x_{\beta^\vee}) = (1-x_{\beta + n \delta})^{-1} x_{\beta^\vee}, \nonumber \\ 
 &\mathbb{S}_q(\ell_n) (x_{\delta^\vee}) = \prod_{k=0}^{n-1}(1-(q^{\frac{1}{2}})^{1-n+2k}x_{\beta + n \delta})^{-1} x_{\delta^\vee}, \nonumber \\
 &\mathbb{S}_q(\ell_{\infty})(x_{\beta^\vee}) = x_{\beta^\vee}.
\end{align} Here we omit to write the injective ring homomorphism $I$ 
in the formula and will keep doing so if there is no possible confusion.

As for  $\mathbb{S}_q(\ell_{\infty})(x_{\delta^\vee})$, Lemma \ref{autosql} does not directly apply since it involves an infinite product over the 
active classes.  Nonetheless we can still write down a formal expression,
\begin{align}\label{sql4}
 \mathbb{S}_q(\ell_{\infty})(x_{\delta^\vee}) = \prod_{m\geq 1} \prod_{k=0}^{m-1} \big( 
 (1-(q^{\frac{1}{2}})^{2-m+2k}x_{m \delta})
 (1-(q^{\frac{1}{2}})^{-m+2k}x_{m \delta})
 \big)x_{\delta^\vee},
\end{align} according to the formula in Lemma \ref{autosql}, which converges when $|x_{\delta}|<|q^{1/2}|<1$. 

As in \cite[Section 3.2]{Bri2} we also need to consider $\mathbb{S}_q(\Delta)$, 
where the sector $\Delta$ contains infinitely many active rays.
Without loss of generality we may take $\Delta$ to be just less than a half-plane 
and have boundary rays in $\Sigma(0)$ and $-\Sigma(0)$.
Then $\mathbb{S}_q(\Delta)$ is given by 
\begin{align}\label{Sql_list}
\mathbb{S}_q(\Delta)(x_{\gamma}) = & \prod_{n \geq 0}\prod_{k=0}^{ |\langle \beta+n\delta,\gamma \rangle |-1} (1-(q^{\frac{1}{2}})^{1-|\langle \beta+n\delta,\gamma \rangle|+2k} x_{\beta+n\delta})^{\text{sgn}\langle \beta+n\delta,\gamma \rangle} \nonumber \\
& \prod_{n \geq 1} \prod_{k=0}^{ |\langle -\beta+n\delta,\gamma \rangle |-1} (1-(q^{\frac{1}{2}})^{1-|\langle -\beta+n\delta,\gamma \rangle|+2k} x_{-\beta+n\delta})^{\text{sgn}\langle -\beta+n\delta,\gamma \rangle} \nonumber \\
& \prod_{m \geq 1} \prod_{k=0}^{m |\langle \delta,\gamma \rangle |-1} \big( 
(1-(q^{\frac{1}{2}})^{2-m|\langle \delta,\gamma \rangle |+2k}x_{m \delta})
(1-(q^{\frac{1}{2}})^{-m|\langle \delta,\gamma \rangle |+2k}x_{m \delta})
\big)^{-\text{sgn}\langle \delta,\gamma \rangle}\cdot x_{\gamma} \ .
\end{align} Again the expression converges when $|x_{\delta}|<|q^{1/2}|<1$.

Next we will adopt the following strategy to formulate the quantum 
Riemann-Hilbert problem for the resolved conifold: we take
Problem \ref{qRHp1} and (\ref{sql123})(\ref{sql4})(\ref{Sql_list}) 
as our starting point and verify the convergence in the resulting problem at last.   

Let $r_n$ be an non-active ray in the sector $\Sigma(n)$. It follows that a solution to the quantum Riemann-Hilbert problem for the resolved conifold is specified by the function $R_{r_n,\gamma_m}(t,q^{\frac{1}{2}})$ in (qRH2) with $\gamma_m = \beta^\vee, \delta^\vee$.
Therefore we define
\begin{equation}\label{BD}
R_{r_n,\beta^\vee}(t,q^{\frac{1}{2}}):=B_n(v,w,t,q^{\frac{1}{2}}), \ R_{r_n,\delta^\vee}(t,q^{\frac{1}{2}}):=D_n(v,w,t,q^{\frac{1}{2}})\ .
\end{equation} 

\begin{remark}
Consider the involution $\sigma: \widehat{\IC_q[\mathbb{T}]} \to \widehat{\IC_q[\mathbb{T}]}$ defined by $\sigma(f(q^{\frac{1}{2}}) x_{\gamma}) = f(q^{\frac{1}{2}}) x_{-\gamma}$. Due to the symmetry of the problem, $\Omega(\gamma)=\Omega(-\gamma)$, we have $$ \mathbb{S}_q(-\ell) \circ \sigma = \sigma \circ  \mathbb{S}_q(\ell). $$ Therefore if we have solved the quantum Riemann-Hilbert problem in one half-plane, we can fill out the whole $\IC^*$ by the symmetry. 
More precisely, at $\theta=0$ we extend the solution by setting
\begin{align}\label{ext}
	R_{-\ell,-\gamma_m}(-t,q^{\frac{1}{2}}) = & R_{\ell,\gamma_m}(t,q^{\frac{1}{2}}), \nonumber \\
	\Phi_{-\ell}(-t) \big(I(x_{-\gamma})\big) = & \text{exp}(-Z(-\gamma_e)/(-t))\  R_{\ell,\gamma_m}(t,q^{\frac{1}{2}}) \cdot x_{-\gamma_m} \in \widehat{\IC_q[\mathbb{T}]} \ .
\end{align}
We also have 
\begin{equation}\label{inv}
	R_{\ell,\gamma_m}(t,q^{\frac{1}{2}}) = R_{\ell,-\gamma_m}(t,q^{\frac{1}{2}})^{-1}, 
\end{equation} by $\Phi_{\ell}(t)\big(I(1)\big)=1 \in \widehat{\IC_q[\mathbb{T}]}$.
\end{remark}

Now we work out the conditions of Problem \ref{qRHp1} imposed on the function $B_n(v,w,t,q^{\frac{1}{2}})$ and $D_n(v,w,t,q^{\frac{1}{2}})$ in (\ref{BD}), with the BPS automorphisms $\mathbb{S}_q(\ell)$, $\mathbb{S}_q(\Delta)$ given by (\ref{sql123})(\ref{sql4})(\ref{Sql_list}). If the $(v,w)$ dependence is understood, we simply write $B_n(t,q^{\frac{1}{2}})=B_n(v,w,t,q^{\frac{1}{2}})$ and   $D_n(t,q^{\frac{1}{2}})=D_n(v,w,t,q^{\frac{1}{2}})$.

\begin{problem}[Quantum Riemann-Hilbert problem for the resolved conifold]
	\label{qRHconifold1}
	Fix $(v,w)\in M_+$. Define $x:=\exp(-2\pi iv/ t)$, and $y:=\exp(- 2\pi i w/t).$ For each $n\in\IZ$ find holomorphic functions $B_n(t,q^{\frac{1}{2}})$ and $D_n(t,q^{\frac{1}{2}})$ for $t$ in the region $$\mathcal{V}(n)=\mathcal{H}_{\ell_{n-1}}\cup\mathcal{H}_{\ell_{n}}, $$
	and $\tau \in \mathcal{H}$, satisfying the following properties.
	\begin{itemize}
		\item[(i)] As $t\to 0$ in any closed subsector of $\mathcal{V}(n)$ one has  
		\begin{equation}
		B_n(t,q^{\frac{1}{2}})\to  1, \qquad  D_n(t,q^{\frac{1}{2}})\to  1.
		\end{equation}
		\item[(ii)] For each $n\in\IZ$ there exists $k>0$ such that for  any closed subsector of $\mathcal{V}(n)$
		\begin{equation} 
			|t| ^{-k} <  |B_n(t,q^{\frac{1}{2}})|,  |D_n(t,q^{\frac{1}{2}})| < |t|^k, \qquad  |t| \gg 0.
		\end{equation}
		\item[(iii)] On the intersection $\mathcal{H}_{\ell_{n}} =\mathcal{V}(n)\cap \mathcal{V}(n+1)$ there are relations
		\begin{equation}\label{BDwallcrossing} B_{n+1}(t,q^{\frac{1}{2}})=B_{n}(t,q^{\frac{1}{2}})\cdot (1-x y^{n} )^{-1}, \ \ D_{n+1}(t,q^{\frac{1}{2}})=D_{n}(t,q^{\frac{1}{2}}) \cdot 
		\prod_{k=0}^{n-1} (1- (q^{\frac{1}{2}})^{1-n+2k}x y^{n})^{-1}.
		\end{equation}
		
		\item[(iv)] In the region $-i\cdot \Sigma(0)$ there are relations

		\begin{equation} \label{BBproduct}
			B_0(t,q^{\frac{1}{2}})\cdot B_0(-t,q^{\frac{1}{2}})=\prod_{n\geq 0}\big(1-x y^{n}\big)\cdot \prod_{n\geq 1}\big(1-x^{-1} y^{n})^{-1},
		\end{equation}

		\begin{align} \label{DDproduct}
		D_0(t,q^{\frac{1}{2}})\cdot D_0(-t,q^{\frac{1}{2}})= & \prod_{n\geq 1}\prod_{k=0}^{n-1}\big(1- (q^{\frac{1}{2}})^{1-n+2k} xy^{n}\big) \cdot \nonumber \\
		& \prod_{n\geq 1} \prod_{k=0}^{n-1}\big(1- (q^{\frac{1}{2}})^{1-n+2k} x^{-1}y^{n}\big)\cdot \nonumber \\
		&\prod_{n \geq 1} \prod_{k=0}^{n-1} \big( 
		(1-(q^{\frac{1}{2}})^{2-n+2k}y^n)
		(1-(q^{\frac{1}{2}})^{-n+2k}y^n)\big)^{-1}\ .
		\end{align}

%		\begin{equation}
%		x=\exp(-2\pi iv/ t), \qquad y=\exp(- 2\pi i w/t).
%		\end{equation}
	\end{itemize}
\end{problem}
\begin{remark}
	Although we have the $\theta(\gamma_e)$ dependence in (\ref{qRH2}) in (qRH2), for simplicity we set $\theta=0$ 
	when deriving these conditions for the resolved conifold, and also for the rest of the paper. The analysis for  $\theta \neq 0$ case is more involved but can be carried out in a completely analogous way. For example, when $\theta \neq 0$ the variables $x=\exp(-2\pi iv/ t)$ 
	and $y=\exp(- 2\pi i w/t)$ in Problem \ref{qRHconifold1} need to be 
	modified by the coordinates of $\theta$, and a new set of conditions as those of Problem \ref{qRHconifold1} can be generated.
		
	Problem \ref{qRHconifold1} (i)(ii) are derived from (qRH2) and (qRH3) respectively. Problem \ref{qRHconifold1} (iii) comes from the jumping condition (qRH1).
	Problem \ref{qRHconifold1} (iv) is obtained by composing all the $\mathbb{S}_q(l)$ in the sector $\Delta$ with boundary rays in $\Sigma(0)$ and $-\Sigma(0)$ and then using relations (\ref{Sql_list})(\ref{ext})(\ref{inv}).

	When solving $B_n(t,q^{\frac{1}{2}})$ and $D_n(t,q^{\frac{1}{2}})$ for all $\tau \in \mathcal{H}$ is not possible, we are allowed to restrict $\tau$ to an open set in $\mathcal{H}$.  

	As the last step we verify the convergence of the expressions with infinite products in Problem \ref{qRHconifold1} (iv). In the region $-i\cdot \Sigma(0)$, the real part of $2\pi i w/ t$ is positive. So we have $|y|<1$ and the RHS of (\ref{BBproduct}) converges. For the RHS of (\ref{DDproduct}) to converge, we need to impose $|y|<|q^{1/2}|<1$, which holds if $\text{Im}(\tau)$ is sufficiently small. 
\end{remark}

\section{A solution to quantum Riemann-Hilbert problem for the resolved conifold}\label{sol}
\subsection{Multiple sine functions}
In \cite{Bri2} the solution to the conifold Riemann-Hilbert problem was given in terms of double sine function, and triple sine function with two equal parameters, with certain exponential prefactors to fix the prescribed asymptotic behaviors. In the quantum case we will need triple sine functions with unequal parameters. It is worthwhile mentioning that the construction here is analogous to the process in which the solution in \cite{BBS} quantizes or deforms that in \cite{Barb}.

Multiple sine functions are defined using the multiple gamma function of Barnes \cite{Barn}. Let $z \in \IC$ and $\omega_1,\cdots,\omega_r\in\IC^*$. 
We have 
\begin{equation}
\text{sin}_r(z\,|\,\omega_1,\cdots,\omega_r)=\Gamma_r(z\,|\, \omega_1,\cdots,\omega_r)^{-1} \ \Gamma_r\big(\sum_{i=1}^{r}\omega_i-z\,|\,\omega_1,\cdots,\omega_r\big)^{(-1)^r} \ .
\end{equation}

For our purpose we will need three parameters $\omega_1, \widetilde\omega_1,\omega_2\in \IC^*$. Let $\overline\omega_1=(\omega_1+\widetilde\omega_1)/2$ and $\Delta\omega_1=(\omega_1-\widetilde\omega_1)/2$.

We define 
\begin{align}
\label{FG}
	F(z\,|\,\overline\omega_1,\omega_2)=\text{exp}\big(-\frac{\pi i}{2}B_{2,2}(z\,|\,\overline\omega_1,\omega_2 )\big)\ \text{sin}_2 (z\,|\,\overline\omega_1,\omega_2), \\
\label{FG2}
	G(z\,|\,\omega_1,\widetilde\omega_1,\omega_2)= \text{exp}\big(\frac{\pi i}{6}B_{3,3}(z+\overline{\omega}_1 \,|\,\omega_1,\widetilde{\omega_1},\omega_2 )\big)\ \text{sin}_3 (z+\overline{\omega}_1 \,|\,\omega_1,\widetilde{\omega_1},\omega_2), 
\end{align} where $B_{2,2}(z\,|\,\overline\omega_1,\omega_2 )$ and $B_{3,3}(z+\overline{\omega}_1 \,|\,\omega_1,\widetilde{\omega_1},\omega_2 )$ are the multiple Bernoulli polynomials, defined by the expansion. 

\begin{equation}
	\frac{s^r e^{zs}}{\prod_{i=1}^{r}(e^{\omega_i s} -1)} = \sum_{n=0}^{\infty} B_{n,r}(z\,|\,\omega_1,\cdots,\omega_r) \frac{s^n}{n!}.
\end{equation}

Here are some multiple Bernoulli polynomials, which will be needed later.
\begin{align}
	& B_{0,2}(z\,|\,\omega_1,\omega_2)=\frac{1}{\omega_1\omega_2}, \qquad B_{1,2}(z\,|\,\omega_1,\omega_2)=\frac{z}{\omega_1\omega_2} - \frac{\omega_1+\omega_2}{2\omega_1\omega_2}, \nonumber \\
	& B_{2,2}(z\,|\,\omega_1,\omega_2)=\frac{z^2}{\omega_1\omega_2} -\bigg(\frac{1}{\omega_1}+\frac{1}{\omega_2}\bigg)z+\frac{1}{6}\bigg(\frac{\omega_2}{\omega_1}+\frac{\omega_1}{\omega_2}\bigg)+\frac{1}{2}.
\end{align}

For later use, the Bernoulli polynomials are defined by

\begin{equation}
	\frac{s\, e^{zs}}{e^s-1}=\sum_{n=0}^{\infty} B_n(z) \frac{s^n}{n!}.
\end{equation}

The prefactors in (\ref{FG})(\ref{FG2}) are chosen such that $F$ and $G$ are admitted to have integral representations without the prefactors (see (\ref{fint})(\ref{gint}) below).

Define the following notation:
\begin{equation}
	\label{bore}
	\begin{aligned}
		x_1=\exp(2\pi i z/\overline\omega_1), &\quad x_2=\exp(2\pi i z/\omega_2), \\ q_1=\exp(2\pi i{\omega_2}/{\overline\omega_1}), &\quad q_2=\exp(2\pi i{\omega_1}/{\omega_2}) , \quad \widetilde{q}_2=\exp(2\pi i{\widetilde\omega_1}/{\omega_2}) .\end{aligned}
\end{equation}

The following two propositions are generalizations of  \cite[Proposition 4.1]{Bri2} and \cite[Proposition 4.2]{Bri2} to the case of $F$ and $G$ with unequal parameters.

\begin{proposition} 
	\label{F}
	The function $F(z\,|\,\overline\omega_1,\omega_2)$ is a single-valued meromorphic function of variables $z\in \IC$ and $\overline\omega_1,\omega_2\in \IC^*$ under the assumption that $\overline\omega_1$ and $\omega_2$ lie on the same side of some straight line through the origin. We have
	\begin{itemize}
		\item[(i)] The function is regular and nonvanishing
		except at the points
		\[z=a\overline\omega_1+b\omega_2, \quad a,b\in\IZ,\]
		which are zeroes if $a,b\leq 0$, poles if $a,b>0$, and otherwise neither.
		
		\item[(ii)] We have the following two difference relations for $F(z\,|\,\overline\omega_1,\omega_2)$:
		\begin{equation}\label{difF}
			\frac{F(z+\overline\omega_1\,|\,\overline\omega_1,\omega_2)}{F(z\,|\,\overline\omega_1,\omega_2)}=\frac{1}{1-x_2}, \qquad \frac{F(z+\omega_2\,|\,\overline\omega_1,\omega_2)}{F(z\,|\, \overline\omega_1,\omega_2)}=
			\frac{1}{1-x_1}.\end{equation}
		
		\item[(iii)] There is a product expansion
		\begin{equation} 
			F(z\,|\,\overline\omega_1,\omega_2)=\prod_{k\geq 1} (1-x_1 q_1^{-k})^{-1}\cdot \prod_{k\geq 0} (1-x_2 (q_2 \widetilde{q}_2)^{k/2}),
		\end{equation}	
		when  $\text{Im}(\overline\omega_1/\omega_2)>0$.

		\item[(iv)] When $\text{Re}(\overline\omega_1)>0$, $\text{Re}(\omega_2)>0$ and $0<\text{Re}(z)<\text{Re}(\overline\omega_1+\omega_2)$ there is an integral representation \begin{equation}\label{fint}
			F(z\,|\,\overline\omega_1,\omega_2)=\exp\Bigg(\int_C \frac{e^{zs}}{(e^{\overline\omega_1 s}-1)(e^{\omega_2 s}-1)}\frac{ds}{s}\Bigg),\end{equation}
		where the contour $C$ is along the real axis from $-\infty$ to $+\infty$ with a small semi-circle around the origin in the upper half-plane.
	\end{itemize}
\end{proposition}
\begin{proof} The proposition is a generalization of \cite[Proposition 4.1]{Bri2}.
The general properties of the double sine functions and $F(z\,|\,\overline\omega_1,\omega_2)$, including part $(i)$ and $(ii)$, are proved in \cite[Appendix A]{JimMiw}. 

The product expansion $(iii)$ is proved in \cite[Corollary 6]{Naru}. The integral representation $(iv)$ is proved in \cite[Proposition 2]{Naru}. 

The difference relation $(ii)$ can also be proved using the integral formula $(iv)$.
\end{proof}

\begin{proposition}
	\label{G} The function $G(z\,|\,\omega_1,\widetilde{\omega}_1,\omega_2)$ is a single-valued holomorphic function of variables $z\in\IC$ and $\omega_1,\widetilde{\omega}_1,\omega_2\in \IC^*$ under the assumption that $\omega_1$, $\widetilde\omega_1$, and $\omega_2$ all lie on the same side of some straight line through the origin. 
	We have the following properties:
	\begin{itemize}
		\item[(i)] The function is entire and vanishes only at the points
		\[z=a\omega_1+b\widetilde\omega_1+c\omega_2, \quad a,b,c\in\IZ,\]
		with  $a,b,c\leq 0$, or $a,b,c> 0$.
			
		\item[(ii)] We have the difference relations
		\begin{align}
			\label{difG1}\frac{G(z+\omega_1\,|\,\omega_1,\widetilde\omega_1,\omega_2)}{G(z\,|\,\omega_1,\widetilde\omega_1,\omega_2)}=&F(z+\overline\omega_1\,|\,\widetilde\omega_1,\omega_2)^{-1} , \\
			\label{difG2}\frac{G(z+\widetilde\omega_1\,|\,\omega_1,\widetilde\omega_1,\omega_2)}{G(z\,|\,\omega_1,\widetilde\omega_1,\omega_2)}=&F(z+\overline\omega_1\,|\,\omega_1,\omega_2)^{-1}.
		\end{align}		
		
		\item[(iii)] When $\text{Re}(\omega_1),\text{Re}(\widetilde\omega_1),\text{Re}(\omega_2)>0$ and $-\text{Re}(\overline\omega_1)< \text{Re}(z)<\text{Re}(\overline\omega_1+\omega_2)$ there is an integral representation \begin{equation}\label{gint}
		G(z\,|\,\omega_1,\widetilde\omega_1,\omega_2)=\exp\Bigg(\int_C \frac{-e^{(z+\overline\omega_1)s}}{(e^{\omega_1 s}-1)(e^{\widetilde\omega_1 s}-1)(e^{\omega_2 s}-1)}\frac{ds}{s}\Bigg),\end{equation}
		where the contour $C$ is the real axis from $-\infty$ to $+\infty$ avoiding the origin by a small detour in the upper half-plane.
	\end{itemize}
\end{proposition}

\begin{proof}
The proposition is a generalization of \cite[Proposition 4.2]{Bri2}.
The general properties $(i)$ and $(ii)$ are covered in \cite[Appendix A]{JimMiw}. The integral formula $(iii)$ follows from \cite[Proposition 2]{Naru}. Alternatively the difference relations $(ii)$ can also be proved by using integral formula.
\end{proof}

\begin{proposition}
	\label{reflection}
	When $\text{Im}(\omega_1/\omega_2)>0$, $\text{Im}(\widetilde\omega_1/\omega_2)>0$ and $z\in \IC$ the following relations hold
	\begin{equation}\label{FF1}
		F(z+\omega_2\,|\,\overline\omega_1,\omega_2)\cdot F(z\,|\,\overline\omega_1,-\omega_2)=  \prod_{k\geq 0} \big(1-x_2 (q_2 \widetilde{q}_2)^{k/2}\big)\cdot \prod_{k\geq 1} \big(1-x_2^{-1} (q_2 \widetilde{q}_2)^{k/2}\big)^{-1},\end{equation}
	\begin{equation}\label{GG1}
		G(z+\omega_2\,|\,\omega_1,\widetilde\omega_1,\omega_2)\cdot G(z\,|\,\omega_1,\widetilde\omega_1,-\omega_2)= \prod_{k_1\geq 0, k_2\geq 0} \big(1-x_2 q_2^{k_1+\frac{1}{2}}\widetilde{q}_2^{k_2+\frac{1}{2}}\big)\cdot\prod_{k_1\geq 0, k_2\geq 0} \big(1-x_2^{-1} q_2^{k_1+\frac{1}{2}}\widetilde{q}_2^{k_2+\frac{1}{2}}\big).\end{equation}
\end{proposition} 

\begin{proof} The proposition is a generalization of \cite[Proposition 4.3]{Bri2}.
First of all $\omega_1$, $\widetilde\omega_1$, and $\omega_2$ all lie on the same side of some straight line through the origin, by the assumption $\text{Im}(\omega_1/\omega_2)>0$, and $\text{Im}(\widetilde\omega_1/\omega_2)>0$. 
Recall the homogeneity property of the multiple sine function \cite[Proposition 2(i)]{Naru},
\begin{equation}
\text{sin}_r(cz\,|\,c\omega_1,\cdots,c\omega_r) = \text{sin}_r(z\,|\,\omega_1,\cdots,\omega_r), \qquad \forall c \in \IC^*.
\end{equation}
This homogeneity property is also enjoyed by the functions 
$F(z\,|\,\overline\omega_1,\omega_2)$ and $G(z\,|\,\omega_1,\widetilde\omega_1,\omega_2)$ since 
\begin{equation}
	B_{n,r}(cz\,|\,c\omega_1,\cdots,c\omega_r) = c^{n-r} B_{n,r}(z\,|\,\omega_1,\cdots,\omega_r), \qquad \forall c \in \IC^*.
\end{equation}

Since we can rotate $\omega_1$, $\widetilde\omega_1$ and $\omega_2$ using the homogeneity property, we may assume $\omega_1$, $\widetilde\omega_1$ are in the right half-plane, and $\omega_2$ almost lies along negative imaginary line with a small positive real part.
Namely, $\text{arg}(\omega_2)= -\pi/2 +\epsilon_+$.

In order to apply the integral formula (\ref{fint}) for $F(z+\omega_2\,|\,\overline\omega_1,\omega_2)$, we require that 
\begin{equation}
0<\text{Re}(z+\omega_2)<\text{Re}(\overline\omega_1+\omega_2).
\end{equation} 

Therefor we have 
\begin{equation}\label{F1}
F(z+\omega_2\,|\,\overline\omega_1,\omega_2)=\exp\Bigg(\int_{C} \frac{e^{(z+\omega_2)s}}{(e^{\overline\omega_1 s}1-1)(e^{\omega_2 s}-1)}\frac{ds}{s}\Bigg) , \qquad 0<\text{Re}(z)<\text{Re}(\overline\omega_1).
\end{equation}

As for $F(z\,|\,\overline\omega_1,-\omega_2)$, we first choose $c= e^{i(-2\epsilon_+)} \in \IC^*$ to rotate $-\omega_2$ into the right half plane. 
We obtain
\begin{align}\label{F2}
F(z\,|\,\overline\omega_1,-\omega_2) =& F(cz\,|\,c\overline\omega_1,-c\omega_2)=\exp\Bigg(\int_C \frac{e^{czs}}{(e^{c\overline\omega_1 s}-1)(e^{-c\omega_2 s}-1)}\frac{ds}{s}\Bigg) \nonumber \\
=& \exp\Bigg(\int_C \frac{e^{z(cs)}}{(e^{\overline\omega_1(cs) }-1)(e^{-\omega_2 (cs)}-1)}\frac{d(cs)}{cs}\Bigg) \nonumber \\
=& \exp\Bigg(\int_{c \cdot C} \frac{e^{zs}}{(e^{\overline\omega_1 s }-1)(e^{-\omega_2 s}-1)}\frac{ds}{s}\Bigg),  0<\text{Re}(cz)<\text{Re}\big(c(\overline\omega_1-\omega_2)\big).
\end{align} 

Combining (\ref{F1}) and (\ref{F2}) and choosing $z$ to lie in the common valid region, we have
\begin{equation}
F(z+\omega_2\,|\,\overline\omega_1,\omega_2)F(z\,|\,\overline\omega_1,-\omega_2)=\exp\Bigg(\int_{C} \frac{e^{(z+\omega_2)s}}{(e^{\overline\omega_1 s}-1)(e^{\omega_2 s}-1)}\frac{ds}{s} + \int_{c\cdot C} \frac{e^{zs}}{(e^{\overline\omega_1 s }-1)(e^{-\omega_2 s}-1)}\frac{ds}{s} \Bigg). 
\end{equation} 
Since the integrands differ by a minus sign, the expression in the exponential picks up the residues at the point $s=2\pi i m/\omega_2$ for $m \in \IZ \setminus \{0\}$, weighted by the signature of $m$.

We obtain 
\begin{align}
& F(z+\omega_2\,|\,\overline\omega_1,\omega_2)F(z\,|\,\overline\omega_1,-\omega_2) = \exp \big( \sum_{m\neq 0} 2\pi i \,\text{Res}_{s=\frac{2\pi i m}{\omega_2}} \frac{e^{(z+\omega_2)s}}{(e^{\overline\omega_1 s}-1)(e^{\omega_2 s}-1)s} \cdot \text{sgn}(m) \big) \nonumber \\
& = \exp \big( \sum_{m \geq 1} \frac{-x_2^m}{m(1-(q_2 \widetilde{q}_2)^{m/2})} + \sum_{m \geq 1} \frac{x_2^{-m}(q_2 \widetilde{q}_2)^{m/2}}{m(1-(q_2 \widetilde{q}_2)^{m/2})} \big) \nonumber \\
& = \prod_{k\geq 0} \big(1-x_2 (q_2 \widetilde{q}_2)^{k/2}\big)\cdot \prod_{k\geq 1} \big(1-x_2^{-1} (q_2 \widetilde{q}_2)^{k/2}\big)^{-1}.
\end{align}
So we have shown that in the allowed open dense region of $z$ for the integral formula of $F$ to hold, $F(z+\omega_2\,|\,\overline\omega_1,\omega_2)F(z\,|\,\overline\omega_1,-\omega_2)$ is equal to the infinite product on the right hand side in (\ref{FF1}), which is an analytic function for $\text{Im}(\omega_1/\omega_2)>0$, $\text{Im}(\widetilde\omega_1/\omega_2)>0$. Therefore they are equal for all $z\in \IC$ by analytic continuation.

Applying the identical argument to prove (\ref{GG1}), we obtain
\begin{align}
& G(z+\omega_2\,|\,\omega_1,\widetilde\omega_1,\omega_2)\cdot G(z\,|\,\omega_1,\widetilde\omega_1,-\omega_2)  \nonumber \\
&= \exp\Bigg(\int_{C} \frac{-e^{(z+\overline\omega_1+\omega_2)s}}{(e^{\omega_1 s}-1)(e^{\widetilde\omega_1 s}-1)(e^{\omega_2 s}-1)}\frac{ds}{s} + 
\int_{c \cdot C} \frac{-e^{(z+\overline\omega_1)s}}{(e^{\omega_1 s}-1)(e^{\widetilde\omega_1 s}-1)(e^{-\omega_2 s}-1)}\frac{ds}{s} \Bigg) \nonumber \\
&= \exp \Bigg( \sum_{m\neq 0} 2\pi i \,\text{Res}_{s={2\pi i m}{\omega_2}} \frac{-e^{(z+\overline\omega_1+\omega_2)s}}{(e^{\omega_1 s}-1)(e^{\omega_1 s}-1)(e^{\omega_2 s}-1)s} \cdot  \text{sgn}(m)\Bigg) \nonumber \\
&= \exp \Bigg( \sum_{m\geq 1} \frac{-x_2^m q_2^{m/2} \widetilde{q}_2^{m/2}}{m(1-q_2^m)(1-\widetilde{q}_2^m)} + \sum_{m\geq 1} \frac{-x_2^{-m} q_2^{m/2} \widetilde{q}_2^{m/2}}{m(1-q_2^m)(1-\widetilde{q}_2^m)} \Bigg) \nonumber \\
&= \prod_{k_1\geq 0, k_2\geq 0} \big(1-x_2 q_2^{k_1+\frac{1}{2}}\widetilde{q}_2^{k_2+\frac{1}{2}}\big)\cdot\prod_{k_1\geq 0, k_2\geq 0} \big(1-x_2^{-1} q_2^{k_1+\frac{1}{2}}\widetilde{q}_2^{k_2+\frac{1}{2}}\big).
\end{align} 
This completes the proof of (\ref{GG1}).
\end{proof}

 Since later we will set $-\omega_2$ to be $t\in \mathcal{H}_{
 \ell_n}$, we study the asymptotic behaviors of $F$ and $G$ as $\omega_2 \to 0$ and $\infty$.

\begin{proposition}\label{asmp0}
	Fix $z\in \IC$ and $\omega_1, \widetilde\omega_1 \in \IC^*$ with $\text{Re}(\omega_1)>0$, $\text{Re}(\widetilde\omega_1)>0$, $0<\text{Re}(z)<\text{Re}(\overline\omega_1)$, $\text{Im}(z/\omega_1)>0$, and $\text{Im}(z/\widetilde\omega_1)>0$. Then as $\omega_2\to 0$ in any closed subsector $\Sigma$ of the half-plane $\text{Re}(\omega_2)>0$ we have asymptotic expansions
	\begin{equation}\label{Fasymp0}
		\log F(z\,|\,\overline\omega_1,\omega_2)\sim 
		\sum_{k\geq 0} \frac{B_k\cdot \omega_2^{k-1} }{k!}  \cdot f_{k-2}(z, \overline\omega_1),
	\end{equation}
	\begin{equation}\label{Gasmp0}
		\log G(z\,|\,\omega_1,\widetilde\omega_1,\omega_2)\sim  \sum_{k\geq 0} \frac{B_k\cdot \omega_2^{k-1} }{k!} \cdot  g_{k-2}(z,\omega_1,\widetilde\omega_1),
	\end{equation}
where the function $f_{k-2}$ and $g_{k-2}$ are holomorphic functions defined by
	\begin{equation}
		f_{k-2}(z, \overline\omega_1)=\int_C \frac{e^{zs}\,s^{k-2}}{e^{\overline\omega_1 s}-1} \ ds,
	\end{equation} 
	\begin{equation}
		g_{k-2}(z,\omega_1,\widetilde\omega_1) = \int_C
		\frac{-e^{(z+\overline\omega_1)s}\,s^{k-2}}{(e^{\omega_1 s}-1)(e^{\widetilde\omega_1 s}-1) } \ ds ,
	\end{equation} where the contour $C$ is along the real axis from $-\infty$ to $\infty$ avoiding the origin by a small semicircle in the upper half plane. 
\end{proposition}
\begin{proof}
In order to apply the integral representation (\ref{fint}) and (\ref{gint})
we need to require $0<\text{Re}(z)<\text{Re}(\overline\omega_1+\omega_2)$ for $F(z\,|\,\overline\omega_1,\omega_2)$ and $0<\text{Re}(z+\overline\omega_1)<\text{Re}(\omega_1+\widetilde{\omega}_1+\omega_2)$ for $G(z\,|\,\omega_1,\widetilde\omega_1,\omega_2)$, which are equivalent to the assumption $0<\text{Re}(z)<\text{Re}(\overline\omega_1)$ as $\omega_2 \to 0$.

By the integral representation formula (\ref{fint}) and the Laurent expansion
	\begin{equation}
		\frac{1}{e^{\omega_2 s} -1} = \sum_{k\geq 0} \frac{B_k\cdot (\omega_2 s)^{k-1} }{k!} ,
	\end{equation} we have
\begin{align}
	\log F(z\,|\,\overline\omega_1,\omega_2) \sim & \int_C \frac{e^{zs}}{e^{\overline\omega_1 s}-1} \sum_{k\geq 0} \frac{B_k\cdot (\omega_2)^{k-1} (s)^{k-2} }{k!} \, ds \nonumber \\
	= & \sum_{k\geq 0} \frac{B_k\cdot \omega_2^{k-1} }{k!}  \cdot f_{k-2}(z, \overline\omega_1). 
\end{align} The computation for (\ref{Gasmp0}) goes similarly. 
\end{proof}

	We would like to strengthen a bit Proposition \ref{asmp0}. Suppose $\text{Im}(z/\overline{\omega}_1)>0$, $\text{Im}(z/\omega_2)>0$, and $z=r e^{i(\theta+ \pi/2)}$ with $r>0$.
	
	Choosing $c=e^{-i(\theta + \epsilon_+)}$, we have
	\begin{align}
	 \log F(z\,|\,\overline\omega_1,\omega_2) =& \log F(cz\,|\,c\overline\omega_1,c\omega_2) \sim \int_C \frac{e^{czs}}{e^{c\overline\omega_1 s}-1} \sum_{k\geq 0} \frac{B_k\cdot (c\omega_2)^{k-1} (s)^{k-2} }{k!} \, ds \nonumber \\
	& =  \sum_{k\geq 0} \frac{B_k\cdot \omega_2^{k-1} }{k!}  \cdot \int_C \frac{e^{czs}\,(cs)^{k-2}}{e^{c\overline\omega_1 s}-1} \ d(cs).  
	\end{align}
	Therefore we obtain the following proposition.

\begin{proposition}\label{asmp00}
	Fix $z=r e^{i(\theta+ \pi/2)}\in \IC$ and $\omega_1, \widetilde\omega_1 \in \IC^*$ with $\text{Im}(z/\omega_1)>0$, and $\text{Im}(z/\widetilde\omega_1)>0$. Then as $\omega_2\to 0$ in any closed subsector $\Sigma$ of the half-plane $H=\{a\in\IC^*\,|\,\text{Im}(z/a)>0\}$, we have asymptotic expansions
	\begin{equation}\label{Fasymp00}
		\log F(z\,|\,\overline\omega_1,\omega_2)\sim 
		\sum_{k\geq 0} \frac{B_k\cdot \omega_2^{k-1} }{k!}  \cdot f^c _{k-2}(z, \overline\omega_1),
	\end{equation}
	\begin{equation}\label{Gasmp00}
		\log G(z\,|\,\omega_1,\widetilde\omega_1,\omega_2)\sim  \sum_{k\geq 0} \frac{B_k\cdot \omega_2^{k-1} }{k!} \cdot  g^c_{k-2}(z,\omega_1,\widetilde\omega_1),
	\end{equation}
	where $c=e^{-i(\theta + \epsilon_+)}$ and the function $f^c_{k-2}$ and $g^c_{k-2}$ are holomorphic functions defined by
	\begin{equation}\label{fc}
		f^c_{k-2}(z, \overline\omega_1)=\int_{c\cdot C} \frac{e^{zs}\,s^{k-2}}{e^{\overline\omega_1 s}-1} \ ds,
	\end{equation} 
	\begin{equation}\label{gc}
		g^c_{k-2}(z,\omega_1,\widetilde\omega_1) = \int_{c\cdot C}
		\frac{-e^{(z+\overline\omega_1)s}\,s^{k-2}}{(e^{\omega_1 s}-1)(e^{\widetilde\omega_1 s}-1) } \ ds ,
	\end{equation} where the contour $C$ is along the real axis from $-\infty$ to $\infty$ avoiding the origin by a small semicircle in the upper half plane. Notice that $c$ depends on the first argument of the functions $f_{k-2}^c$ and $g_{k-2}^c$.
\end{proposition}	

%\begin{remark}
%	Later we will need the asymptotic expansion for $G\big(\Delta\omega_1\,|\,\omega_1,\widetilde\omega_1,\omega_2\big)$ as $\omega_2\to 0$ in any closed subsector $\Sigma$ of the half-plane $\text{Re}(\omega_2)>0$.
%\end{remark}

\begin{lemma} \cite[Proposition 4.5]{Bri2} \label{Brilemma}
	Let $\omega \in \IC^*$ with $\text{Re}(\omega)>0$. Then for all $d\in\IZ$ the relation 
	\begin{equation}
		-\int_C \frac{e^{\omega s} \cdot s^{1-d}}{(e^{\omega s}-1)^2}\, ds = \frac{(d-1)\zeta(d)}{2 \pi i} \cdot \big(\frac{\omega}{2\pi i}\big)^{d-2},
	\end{equation} holds where the contour $C$ is along the real axis from $-\infty$ to $+\infty$ with a small detour around the origin in the upper half plane. When $d=1$ the factor $(d-1)\zeta(d)$ on right hand side of the formula is understood to be $1$.
\end{lemma} 

\begin{proposition}
	(i) Suppose $z\in\IC$ with $\text{Im}(z/\overline\omega_1)>0$. Then as $\omega_2\to \infty$ in any closed subsector $\Sigma$ of the half-plane $H=\{a\in\IC^*\,|\,\text{Im}(z/a)>0\}$ we have the asymptotic expansion
	\begin{equation}
		\log F(z\,|\,\overline\omega_1,\omega_2)\sim -\frac{\pi i}{12}\cdot \frac{ \omega_2}{\overline\omega_1}+B_1(z/\overline\omega_1)\cdot \log(\omega_2)+O(1).
	\end{equation} 
	(ii) Suppose $z\in\IC$, $\text{Im}(z/\omega_1)>0$, and  $\text{Im}(z/\widetilde\omega_1)>0$. Then as $\omega_2 \to \infty$ in any closed subsector $\Sigma$ of the half-plane $H=\{a\in\IC^*\,|\,\text{Im}(z/a)>0\}$ we have the asymptotic expansion
	\begin{align}\label{Gasympinfty}
	\log G(z\,|\,\omega_1,\widetilde\omega_1,\omega_
	2) \sim\, & \frac{B_{0,2}(z+\overline\omega_1\,|\,\omega_1,\widetilde\omega_1)\zeta(3)}{4\pi^2}\cdot \omega_2^2 - \frac{B_{1,2}(z+\overline\omega_1\,|\,\omega_1,\widetilde\omega_1) \zeta(2)}{2\pi i}\cdot \omega_2 \nonumber \\
	& -\frac{B_{2,2}(z+\overline\omega_1\,|\,\omega_1,\widetilde\omega_1)}{2} \log(\omega_2) + O(1) .
\end{align}
\end{proposition}
\begin{proof}
Part $(i)$ is proved in \cite[Proposition 4.8]{Bri2} with the assumption $\text{Re}(z)>0$ and $\omega_2 \to \infty$ in any closed subsector $\Sigma$ of the half-plane $\text{Re}(\omega_2)>0$. Under our current assumption $\text{Im}(z/\overline\omega_1)>0$ we can rotate $z$ to almost align with the positive imaginary line and have a small real part. All the rotation dependence cancels out in the end of computation and $\omega_2$ is required to lie in a closed subsector $\Sigma$ of the half-plane $H=\{a\in\IC^*\,|\,\text{Im}(z/a)>0\}$.

For part $(ii)$ we again apply a rotation such that $z$ almost aligns along the positive imaginary line and has a small real part. It can be easily checked that the rotation dependence cancels in the end. Therefore without loss of generality, we may assume $\text{Re}(\omega_1)>0$, $\text{Re}(\widetilde\omega_1)>0$, $\text{Re}(z+\overline\omega_1)>0$. 
Applying the integral formula (\ref{gint}) we have 
\begin{equation}
	\log G(z\,|\,\omega_1,\widetilde{\omega}_1,\omega_2) = \int_C \frac{-e^{(z+\overline\omega_1)s}}{(e^{\omega_1 s}-1) (e^{\widetilde\omega_1 s}-1)(e^{\omega_2 s}-1)} \cdot \frac{ds}{s}, 
\end{equation} with the valid region
$0<\text{Re}(z+\overline\omega_1)<\text{Re}(\omega_1+\widetilde\omega_1+\omega_2)$, which becomes $0<\text{Re}(z+\overline\omega_1)$ in the $\omega_2 \to \infty$ limit. 

By differentiating with respect to $\omega_2$ we have
\begin{align}\label{Gprime}
	\frac{\partial}{\partial \omega_2}\log & G(z\,|\,\omega_1,\widetilde{\omega}_1,\omega_2) = \int_C \frac{e^{(z+\overline\omega_1)s}}{(e^{\omega_1 s}-1) (e^{\widetilde\omega_1 s}-1)} \frac{e^{\omega_2 s}}{(e^{\omega_2 s}-1)^2} \cdot ds, \nonumber \\
	\sim & \sum_{k=0}^{\infty} \frac{B_{k,2}(z+\overline\omega_1\,|\,\omega_1,\widetilde\omega_1)}{k!} \int_C s^{k-2} \cdot \frac{e^{\omega_2 s}}{(e^{\omega_2 s}-1)^2} \cdot ds \nonumber \\
	= & \sum_{k=0}^{\infty} \frac{B_{k,2}(z+\overline\omega_1\,|\,\omega_1,\widetilde\omega_1)}{k!} \frac{(k-2)\zeta(3-k)}{2 \pi i} \cdot \big(\frac{\omega_2}{2\pi i}\big)^{1-k}\qquad \text{(by Lemma \ref{Brilemma})}\nonumber \\
	= & -\frac{\zeta(3) \omega_2}{2 \pi^2 \omega_1 \widetilde\omega_1 } -B_{1,2}(z+\overline\omega_1\,|\,\omega_1,\widetilde\omega_1) \frac{\zeta(2)}{2 \pi i} - B_{2,2}(z+\overline\omega_1\,|\,\omega_1,\widetilde\omega_1) \frac{1}{2\omega_2} + \cdots. 
\end{align} 

After integrating (\ref{Gprime}) with respect to $\omega_2$ we obtain (\ref{Gasympinfty}). Also recall that $\zeta(2)=\pi^2/6$.

\end{proof}
\subsection{A solution}
Under the assumption 
\begin{align*}
	& \text{Im}(z/\omega_1)>0, \text{Im}(z/\widetilde\omega_1)>0,  \\
	&\text{Im}(\Delta\omega_1/\omega_1)>0, \text{Im}(\Delta\omega_1/\widetilde\omega_1)>0, 
\end{align*} we define
\begin{align}
	F^*(z\,|\,\overline\omega_1,\omega_2)=&F(z\,|\,\overline\omega_1,\omega_2)\cdot\text{exp}\big(Q_F(z\,|\,\overline\omega_1,\omega_2)\big),\nonumber \\
	G^*(z\,|\,\omega_1,\widetilde\omega_1,\omega_2) =& \frac{G(z\,|\,\omega_1,\widetilde\omega_1,\omega_2)}{G\big(\Delta\omega_1\,|\,\omega_1,\widetilde\omega_1,\omega_2\big)} \cdot \text{exp}\big(Q_G(z\,|\,\omega_1,\widetilde\omega_1,\omega_2)\big),
\end{align} where $Q_F$ and $Q_G$ are Laurent polynomials in $\omega_2$ defined by
\begin{align}
Q_F(z\,|\,\overline\omega_1,\omega_2) =& - f_{-2}^c(z,\overline\omega_1)\cdot\frac{1}{\omega_2} +\frac{f_{-1}^c(z,\overline\omega_1)}{2} + \frac{\pi i}{12} \cdot \frac{\omega_2}{\overline\omega_1}, \nonumber  \\
Q_G(z\,|\,\omega_1,\widetilde\omega_1,\omega_2) =& -\big( g_{-2}^c(z,\omega_1,\widetilde\omega_1)-g_{-2}^c(\Delta\omega_1,\omega_1,\widetilde\omega_1)\big)\cdot\frac{1}{\omega_2} +\frac{1}{2} \big( g_{-1}^c(z,\omega_1,\widetilde\omega_1)-g_{-1}^c(\Delta\omega_1,\omega_1,\widetilde\omega_1)\big) \nonumber \\
&+\big(B_{1,2}(z+\overline\omega_1\,|\,\omega_1,\widetilde\omega_1)-B_{1,2}(\omega_1\,|\,\omega_1,\widetilde\omega_1)\big)\cdot\frac{\zeta(2)\omega_2}{2\pi i}.
\end{align}

\begin{remark}
	Since the definition of $Q_G(z\,|\,\omega_1,\widetilde\omega_1,\omega_2)$ involves $g_{-2}^c(\Delta\omega_1,\omega_1,\widetilde\omega_1)$ and $g_{-1}^c(\Delta\omega_1,\omega_1,\widetilde\omega_1)$, we need to require $\text{Im}(\Delta\omega_1/\omega_1)>0$, and  $\text{Im}(\Delta\omega_1/\widetilde\omega_1)>0$ to ensure their convergences.
\end{remark}

\begin{proposition}\label{FstarGstar}
(i) Let $z\in\IC$ and $\omega_1$, $\widetilde\omega_1$, $\omega_2$ lie on the same side of some straight line through the origin. Also assume $\text{Im}(z/\omega_1)>0, \text{Im}(z/\widetilde\omega_1)>0, \text{Im}(\Delta\omega_1/\omega_1)>0, \text{Im}(\Delta\omega_1/\widetilde\omega_1)>0$.

We have the following difference relations:
\begin{align}
	\label{diffF}\frac{F^{*}(z+\overline\omega_1\,|\,\overline\omega_1,\omega_2)}{F^{*}(z\,|\,\overline\omega_1,\omega_2)} =& \frac{1}{1-x_2}, \\
	\label{diffG1}\frac{G^{*}(z+\omega_1\,|\,\omega_1,\widetilde\omega_1,\omega_2)}{G^{*}(z\,|\,\omega_1,\widetilde\omega_1,\omega_2)} =& \frac{1}{F^{*}(z+\overline\omega_1\,|\,\widetilde\omega_1,\omega_2)},  \qquad \text{if further assume $\text{Im}(\omega_1/\widetilde\omega_1)>0$,}\\
	\label{diffG2}\frac{G^{*}(z+\widetilde\omega_1\,|\,\omega_1,\widetilde\omega_1,\omega_2)}{G^{*}(z\,|\,\omega_1,\widetilde\omega_1,\omega_2)} =& \frac{1}{F^{*}(z+\overline\omega_1\,|\,\omega_1,\omega_2)}, \qquad \text{if further assume $\text{Im}(\widetilde\omega_1/\omega_1)>0$.}
\end{align}

(ii) When $\text{Im}(\omega_1/\omega_2)>0$, $\text{Im}(\widetilde\omega_1/\omega_2)>0$ we have the reflection relations
\begin{equation}\label{FF2}
	F^*(z\,|\,\overline\omega_1,\omega_2)\cdot F^*(z\,|\,\overline\omega_1,-\omega_2)=  \prod_{k\geq 0} \big(1-x_2 (q_2 \widetilde{q}_2)^{k/2}\big)\cdot \prod_{k\geq 1} \big(1-x_2^{-1} (q_2 \widetilde{q}_2)^{k/2}\big)^{-1},\end{equation}
\begin{equation}\label{GG2}
	G^*(z\,|\,\omega_1,\widetilde\omega_1,\omega_2)\cdot G^*(z\,|\,\omega_1,\widetilde\omega_1,-\omega_2)= \prod_{k_1\geq 0, k_2\geq 0} \big(1-x_2 q_2^{k_1+\frac{1}{2}}\widetilde{q}_2^{k_2+\frac{1}{2}}\big)\cdot\prod_{k_1\geq 0, k_2\geq 0} \big(1-x_2^{-1} q_2^{k_1+\frac{1}{2}}\widetilde{q}_2^{k_2+\frac{1}{2}}\big).\end{equation}
\end{proposition}

\begin{proof}
	(\ref{diffF}),(\ref{diffG1}), and (\ref{diffG2}) follows from the difference relations (\ref{difF}),(\ref{difG1}), (\ref{difG2}), and the following identities of the exponential prefactors,
	\begin{align}
		Q_F(z+\overline\omega_1\,|\,\overline\omega_1,\omega_2) =& Q_F(z\,|\,\overline\omega_1,\omega_2), \\
		Q_G(z+\omega_1\,|\,\omega_1,\widetilde\omega_1,\omega_2)-Q_G(z\,|\,\omega_1,\widetilde\omega_1,\omega_2)=& -Q_F(z+\overline\omega_1\,|\,\widetilde\omega_1,\omega_2), \\
		Q_G(z+\widetilde\omega_1\,|\,\omega_1,\widetilde\omega_1,\omega_2)-Q_G(z\,|\,\omega_1,\widetilde\omega_1,\omega_2)=& -Q_F(z+\overline\omega_1\,|\,\omega_1,\omega_2).
	\end{align} The conditions $\text{Im}(\omega_1/\widetilde\omega_1)>0$ for (\ref{diffG1}) and $\text{Im}(\widetilde\omega_1/\omega_1)>0$ for (\ref{diffG2}) come from the convergence regions of $f_{k-2}^c$ and $g_{k-2}^c$ in (\ref{fc})(\ref{gc}).
	
	 The reflection relations in part $(ii)$ come from (\ref{FF1}) and (\ref{GG1}) and the following relations
	 \begin{align}
	 	& \log F(z+\omega_2\,|\,\overline\omega_1,\omega_2) - \log F(z\,|\,\overline\omega_1,\omega_2) = f_{-1}^c(z,\overline\omega_1), \\
		& \log G(z+\omega_2\,|\,\omega_1,\widetilde\omega_1,\omega_2)- \log G(z+\omega_2\,|\,\omega_1,\widetilde\omega_1,\omega_2) = g_{-1}^c(z,\omega_1,\omega_2),
	 \end{align} which can be easily checked by the integral formulas.
	
\end{proof}

We define the following functions.
\begin{align}
	B_0(v,w,t,q^{\frac{1}{2}})= & B_0(v,w,t)=F^*(v\,|\,w,-t), \label{B0} \\
	D_0(v,w,t,q^{\frac{1}{2}})= & G^*(v\,|\,w-t\tau/2,w+t\tau/2,-t), \label{D0} \\
	B_n(v,w,t,q^{\frac{1}{2}})= & B_0(v+nw,w,t), \\
	D_n(v,w,t,q^{\frac{1}{2}})= & D_0(v+nw-nt\tau/2,w,t,q^{\frac{1}{2}}) \cdot \prod_{k=0}^{n-1} B_0(v+nw+(1-n+2k)t\tau/2,w+t\tau/2,t) , \label{Dn} \\
    & q^{\frac{1}{2}}=\text{exp}(\pi i \tau).
\end{align} 

\begin{remark}
	For $B_0(v,w,t,q^{\frac{1}{2}})$ to be well-defined, one needs to impose
	\begin{equation}\label{Bcondition}
		\text{Im}\big(\frac{v}{w}\big)>0, \qquad \text{Im}\big(\frac{v}{-t}\big)>0.
	\end{equation}
	
	For $D_0(v,w,t,q^{\frac{1}{2}})$ to be well-defined, one needs to impose
	\begin{align}\label{Dcondition1}
		\text{Im}\big(\frac{v}{w-t\tau/2}\big)>0, \qquad \text{Im}\big(\frac{v}{w+t\tau/2}\big)>0, \qquad \text{Im}\big(\frac{v}{-t}\big)>0, \nonumber \\
		\text{Im}\big(\frac{-t\tau/2}{w-t\tau/2}\big)>0 , \qquad \text{Im}\big(\frac{-t\tau/2}{w+t\tau/2}\big)>0, \qquad \text{Im}\big(\frac{-t\tau/2}{-t}\big)=\text{Im}(\tau/2)>0 . 
	\end{align}
It is clear that given $v,w,-t \in \IC^*$ satisfying (\ref{Bcondition}) there is an open set in the upper half-plane for $\tau$ to choose from such that conditions (\ref{Dcondition1}) hold. 

We describe this open set. (\ref{Bcondition}) implies that $w$ and $-t$ lie in the same half plane determined by $v$. Let $\tau = \rho e^{i \zeta}$ and take $\rho$ to be sufficiently small such that the denominators $w \pm t \tau /2$ in (\ref{Dcondition1}) 
are almost $w$. Then we have $\text{arg}(-w/t)< \zeta< \text{arg}(-w/t) +\pi$ and $0< \zeta<\pi$ from the second line of (\ref{Dcondition1}). Notice that the closure of this open set contains $\tau=0$. 

The open set arising from the well-definedness consideration of 
$B_n(v,w,t,q^{\frac{1}{2}})$ and $D_n(v,w,t,q^{\frac{1}{2}})$ is denoted by $\mathcal{H}(v,w,t,n) \subset \mathcal{H}$.
\end{remark}

\begin{proposition}\label{BDwallcrossingProp}
	Let $(v,w) \in M_+$ and $t\in\IC^*$ satisfying $\text{Im}\big(\frac{v}{-t}\big)>0$. Then for $\tau \in \mathcal{H}(v,w,t,n) \cap \mathcal{H}(v,w,t,n+1) \subset \mathcal{H}$ such that $B_n(v,w,t,q^{\frac{1}{2}})$, $B_{n+1}(v,w,t,q^{\frac{1}{2}})$, $D_n(v,w,t,q^{\frac{1}{2}})$, and $D_{n+1}(v,w,t,q^{\frac{1}{2}})$ are all well-defined, we have
	\begin{align}
& \label{Bwallcrossing} B_{n+1}(v,w,t,q^{\frac{1}{2}})=B_{n}(v,w,t,q^{\frac{1}{2}})\cdot (1-x y^{n} )^{-1},\\ 
& \label{Dwallcrossing} D_{n+1}(v,w,t,q^{\frac{1}{2}})=D_{n}(v,w,t,q^{\frac{1}{2}}) \cdot 
\prod_{k=0}^{n-1} (1-(q^{\frac{1}{2}})^{1-n+2k}x y^{n})^{-1}, \\
& x=\exp(-2\pi iv/ t), \qquad y=\exp(- 2\pi i w/t). \nonumber 
	\end{align}
\end{proposition}

\begin{proof}
	We prove (\ref{Dwallcrossing}) by direct computation. The proof of (\ref{Bwallcrossing}) is similar. 
	
	By the definition of $D_n(v,w,t,q^{\frac{1}{2}})$ (\ref{Dn}), we have
	
	\begin{align}
		\frac{D_{n+1}(v,w,t,q^{\frac{1}{2}})}{D_n(v,w,t,q^{\frac{1}{2}})} = & \frac{D_0\big(v+(n+1)w-(n+1)t\tau/2,w,t\big)\cdot B_0\big(v+(n+1)w-nt\tau/2,w+t\tau/2,t\big)}{D_0\big(v+nw-nt\tau/2,w,t\big)} \cdot \nonumber \\
		& \prod_{k=0}^{n-1} \frac{B_0\big(v+(n+1)w+(2-n+2k)t\tau/2,w+t\tau/2,t\big)}{B_0\big(v+nw+(1-n+2k)t\tau/2,w+t\tau/2,t\big)} . \label{Dnwallcrossing}
	\end{align} 

Using the definition of $B_0$ and $D_0$ (\ref{B0})(\ref{D0}) and the difference relations (\ref{diffF})(\ref{diffG1}), the equation (\ref{Dnwallcrossing}) becomes
\begin{equation}
	\frac{D_{n+1}(v,w,t,q^{\frac{1}{2}})}{D_n(v,w,t,q^{\frac{1}{2}})} = \prod_{k=0}^{n-1} (1-(q^{\frac{1}{2}})^{1-n+2k}x y^{n})^{-1}.
\end{equation}
\end{proof}

\begin{theorem}
	Fix $(v,w)\in M_+$. Then there is an open set $\mathcal{H}(v,w,t,n)$ in the upper half-plane $\mathcal{H}$ determined by $(v,w),t$ and $n$, such that $B_n(v,w,t,q^{\frac{1}{2}})$ and $D_n(v,w,t,q^{\frac{1}{2}})$ are holomorphic for $t \in \mathcal{H}_{\ell_n}$ and $\tau \in \mathcal{H}(v,w,t,n)$. Moreover they solve Problem \ref{qRHconifold1}, the quantum Riemann-Hilbert problem for the resolved conifold.
\end{theorem}

\begin{proof}
	It is proved in \cite[Theorem 5.2]{Bri2} that $B_0(v,w,t)$ is holomorphic in $\mathcal{V}(0)$. The proof actually only relies on the ordering of $v,w,-t$ in their half-plane. 
	It follows that $B_n(v,w,t)$ is holomorphic in $\mathcal{V}(n)$. 
	
	As for $D_n(v,w,t,q^{\frac{1}{2}})$, we restrict to the open set $\mathcal{H}(v,w,t,n)$ in the upper half-plane, determined by $(v,w),t$ and $n$, such that $D_n(v,w,t,q^{\frac{1}{2}})$ is well-defined. 
	The well-definedness conditions involve the ordering of arguments in the functions $B_0$ and $D_0$ appearing in the $D_n$. 
	
	For example, since there is a factor $D_0(v+nw-nt\tau/2,w,t,q^{\frac{1}{2}})$ appearing in $D_n(v,w,t,q^{\frac{1}{2}})$, we need to impose 
\begin{align}\label{Dcondition}
	\text{Im}\big(\frac{v+nw-nt\tau/2}{w-t\tau/2}\big)>0, \qquad \text{Im}\big(\frac{v+nw-nt\tau/2}{w+t\tau/2}\big)>0, \qquad \text{Im}\big(\frac{v+nw-nt\tau/2}{-t}\big)>0, \nonumber \\
	\text{Im}\big(\frac{-t\tau/2}{w-t\tau/2}\big)>0 , \qquad \text{Im}\big(\frac{-t\tau/2}{w+t\tau/2}\big)>0, \qquad \text{Im}\big(\frac{-t\tau/2}{-t}\big)=\text{Im}(\tau/2)>0 . 
\end{align}	There are similar conditions for the $B_0$ factors appearing in $D_n(v,w,t,q^{\frac{1}{2}})$.
	
	These ensure that the various factors $G(\Delta\omega_1\,|\,\omega',\omega'')$ in the denominator do not attain zero, by the Proposition \ref{G}$(i)$. Therefore part $(i)(ii)$ of Problem (\ref{qRHpconifold}) is satisfied.
	
	Part $(iii)$ of Problem (\ref{qRHpconifold}) is satisfied due to Proposition \ref{BDwallcrossingProp}. Part $(iv)$ of Problem (\ref{qRHpconifold}) follows from Proposition \ref{FstarGstar}(ii).
\end{proof}

Some remarks are in order before concluding this section. As the title of section suggests, we only prove a solution to the quantum Riemann-Hilbert problem exists
since the uniqueness part of the theorem is lacking. It is likely that our formulation of the quantum Riemann-Hilbert problem is not ultimate and still subject to 
improvement. For example, in order to have certain version of uniqueness theorem, one needs to define the equivalence classes of quantum deformations and the equivalence classes of the solutions.  

As already alluded in the introduction, our solution has the feature that the valid region of the quantum parameter $q^{\frac{1}{2}}=\exp(\pi i \tau)$ 
varies on the space of stability conditions and BPS $t$-plane. However in the next section we will see that the refined Chern-Simons theory, which is dual or equivalent to 
refined Donaldson-Thomas theory on resolved conifold, does not have this restriction on the parameter $\tau$.  Therefore it would be 
interesting to understand whether the restriction of $\tau$ is due to the limitation of our approach or it has any physical implication.

\section{Refined Chern-Simons theory and the non-perturbative completion}
Via the large $N$ duality in string theory the correspondence between $SU(N)$ Chern-Simons theory on $S^3$ and the topological string/Gromov-Witten theory on the resolved conifold was established \cite{GV}. 
The correspondence was later promoted to the one between the refined Chern-Simons theory and the refined topological string theory in \cite{AS}.

In \cite{KreMkr} the partition function of the refined Chern-Simons theory is given by 
\begin{equation}\label{CS}
Z(\overline\delta,\overline\mu,\beta)=\frac{\beta}{\sqrt{\overline\mu - \frac{1}{2}(1-\beta)}} \frac{\text{sin}_3\big( \frac{1}{2}(\sqrt{\beta}+
	\frac{1}{\sqrt{\beta}})+\overline\delta \overline\mu\,|\, \frac{1}{\sqrt{\beta}}, \sqrt{\beta},\overline\delta \big)}
	{\text{sin}_3\big( \sqrt{\beta}\,|\, \frac{1}{\sqrt{\beta}}, \sqrt{\beta},\overline\delta \big)},
\end{equation} where $\overline\delta$ is the effective coupling constant of the Chern-Simons theory, $\overline\mu$ is proportional to $N$, and $\beta$ is a quantity related to the so-called $\Omega$-background, introduced by Nekrasov \cite{Nek}. 

Under the parameter redefinition,
\begin{equation} 
g_s=2 \pi i / \overline\delta, \qquad Q=\exp(-2\pi i \overline\mu),
\end{equation} 
the partition function can be written as
\begin{equation}
	Z(\overline\delta,\overline\mu,\beta) = \exp(F^{P}\big(g_s,Q,\beta)+F^{NP}(g_s,Q,\beta)\big).
\end{equation} 

Here $F^P\big(g_s,Q,\beta)$ is the perturbative free energy of the refined topological string on the resolved conifold, containing the perturbative genus expansion in $g_s$ and constant map contributions, while $F^{NP}(g_s,Q,\beta)$ is the non-perturbative free energy, containing contributions which go like
\begin{equation}
	\exp\big(-\frac{2n\pi^2}{\sqrt{\beta}g_s} \big),\qquad \text{or} \qquad \exp\big(-\frac{2n\pi^2 \sqrt{\beta}}{g_s} \big).
\end{equation}

In other words, the partition function (\ref{CS}) can be regarded as the non-perturbative completion of refined topological string/Gromov-Witten theory on the resolved conifold. 

Ignoring the exponential prefactor $\exp\big(Q_G(z\,|\,\omega_1,\widetilde\omega_1,\omega_2)\big)$ in $D_0(v,w,t,q^{1/2})$ (\ref{D0}),
we have
\begin{equation}\label{D0sin3}
	D_0(v,w,t,q^{1/2}) \sim \frac{G(v\,|\,w-t\tau/2,w+t\tau/2,-t)}{G(-t\tau/2\,|\,w-t\tau/2,w+t\tau/2,-t)} \sim \frac{\text{sin}_3(v+w\,|\,w-t\tau/2,w+t\tau/2,-t)}
	{\text{sin}_3(w-t\tau/2\,|\,w-t\tau/2,w+t\tau/2,-t)},
\end{equation} where in the second step the prefactors of multiple Bernoulli polynomials are also omitted.

Comparing (\ref{CS}) and (\ref{D0sin3}) we arrive at the following parameter identifications,
\begin{equation}
	\overline\delta\overline\mu \longleftrightarrow v, \qquad 
	\sqrt{\beta} \longleftrightarrow w-t\tau/2, \qquad
	\frac{1}{\sqrt{\beta}} \longleftrightarrow w+t\tau/2, \qquad
	\overline\delta \longleftrightarrow -t.
\end{equation} The upshot of this analysis is that the solution to the quantum Riemann-Hilbert problems is reminiscent of the non-perturbative completion of the refined Donaldson-Thomas theory/Gromov-Witten theory, at least on the resolved conifold.
Therefore solving the quantum Riemann-Hilbert problems provides a possible non-perturbative definition for the refined Donaldson-Thomas theory. 
 
\bibliographystyle{alpha}

\end{document}